\documentclass[12pt]{amsart} %

\usepackage{float}

\usepackage{epsfig}
\usepackage{tikz}
\usepackage{array,amsmath, enumerate,  url, psfrag}
\usepackage{amssymb, amsaddr, fullpage}
\usepackage{graphicx,subfigure}
\usepackage{color}
\restylefloat{figure}

\newtheorem{thm}{Theorem}[section]
\newtheorem{lem}[thm]{Lemma}

\newtheorem{definition}[thm]{Definition}
\newtheorem{example}[thm]{Example}

\title{The odd chromatic number of a toroidal graph is at most 9}

\author{Fangyu Tian$^{1}$\hskip 0.2in  yuxue Yin$^{2}$}

\address{
$^{1}$\small Department of Mathematics, Central China Normal University, Wuhan, Hubei, China.\\
$^2$\small Department of EE, Tsinghua University, Beijing, China.
}

\email{yinyuxue945@mail.tsinghua.edu.cn}

\begin{document}
\maketitle
\begin{abstract}
It's well known that every planar graph is $4$-colorable.
A toroidal graph is a graph that can be embedded on a torus. 
It's proved that every toroidal graph is $7$-colorable.
A proper coloring of a graph is called \emph{odd} if every non-isolated vertex has at least one color that appears an odd number of times in its neighborhood. The smallest number of colors that admits an odd coloring of a graph $ G $ is denoted by $\chi_{o}(G)$.
In this paper, we prove that if $G$ is tortoidal, then $\chi_{o}\left({G}\right)\le9$; Note that $K_7$ is a toroidal graph, the upper bound is no less than $7$. 
\end{abstract}

\section{Introduction}

In this paper, all graphs are finite and simple, which means no parallel edges and no loops
at their vertices.  A proper $k$-coloring of a simple graph $G$ is an assignment (or a special case of labeling) of $k$ colors to the vertices of $G$ so that no two adjacent vertices share the same color. We say a graph $G$ is $c$-colorable if it admits a proper $c$-coloring. The chromatic number of a graph $G$ is the minimum $c$ such that $G$ is $c$-colorable, and this minimum color is denoted by $\chi(G)$. 
It's well known that $ \chi(G)\le4 $ if $G$ is planar. 
For a proper coloring, there exists at least one color that appears an odd number of times in the neighborhood of $v$, then we say $v$ admits an odd coloring. We use $c_o(v)$ to denote the color and $C_o(v)$ to denote the set of the odd colors. An odd $c$-coloring of a graph is a proper $c$-coloring with the additional constraint that each vertex admits an odd coloring. A graph G is odd $c$-colorable if it has an odd $c$-coloring. The odd chromatic number of a graph G, denoted by $\chi_o(G)$, is the minimum $c$ such that G has an odd $c$-coloring. Odd coloring has potential applications in many areas, for example, battery consumption aspects of sensor networks and in RFID protocols~\cite{smorodinsky2013conflict}.

Odd coloring was introduced very recently by Petru$\breve{s}$evski and $\breve{S}$krekovski\\~\cite{petruvsevski2021colorings}, who proved that planar graphs are odd $9$-colorable. Note that a 5-cycle is a planar graph whose odd chromatic number is exactly 5,  they further conjectured that planar graphs are odd $5$-colorable. Petr and Portier~\cite{petr2022odd} proved that planar graphs are odd $8$-colorable. Fabrici\cite{fabrici2022proper} proved a strengthening version about planar graphs regarding similar coloring parameters.  
 Cranston~\cite{cranston2022odd}
studied the restriction of girth and proved that, $\chi_{o}(G)\leq5$ if $G$ is a planar graph with girth at least $7$, and $\chi_{o}(G)\leq6$ if $G$ is a planar graph with girth at least $6$.
Eun-Kyung Cho~\cite{cho2022odd} focused on a sparse graph and conjectured that, for $c\ge4$, if G is a graph with $mad(G)\le \frac{4c-4}{c+1}$, then $\chi_o(G)\le c$. And proved that, if G is a graph with $mad(G)\le mad(K^*_{c+1})= \frac{4c}{c+2}$, then $\chi_o(G)\le c$ for $c\ge 7$, unless G contains $K^*_{c_1}$ as a subgraph. They further proved that $\chi_{o}(G)\leq5$ if $G$ is a planar graph with girth at least $6$.

Suppose that $G$ is a toroidal graph. The $7$-color theorem~\cite{kauffman2009seven} shows that $\chi(G)\le7$. Notice that $K_7$ is a toroidal graph, $\chi(G)=7$. Note that $\chi_o(G)\ge\chi(G)$, if $G$ is a toroidal graph, then $\chi_o(G)\ge7$.
We proved that,
\begin{thm}\label{th1}
If $G$ is a toroidal graph, then $\chi_o(G)\le9$
\end{thm}

We prove Theorem~\ref{th1} by reduction. In the construction of the minimal counterexample, we organize the constraints in a creative way, which simplifies our proof greatly and can be modified to settle other coloring problems. Moreover, we pay a lot attention to summarizing the complex situations in the proof and present it in brevity. More precisely, we force the most difficult part into the configuration as is shown in Figure 1 via discharging method, and simplify the analysis of this main reducible configuration by splitting it into Lemmas~\ref{32123},~\ref{42123} and~\ref{1234}, and then come to a conclusion in Lemma~\ref{6-v} based on the former tool Lemmas and Claim.

\section{Proof}
Let $G$ be a counterexample to Theorem~\ref{th1} with the minimum number of $4^+$-vertices, and subject to that, the number of $5^+$-neighbors of $5^+$-vertex $G$ is minimized, and subject to these conditions $|E(G)|$ is minimized.

\begin{lem}\label{minimum degree}
The minimum degree $\delta(G)\ge5$
\end{lem}
\begin{proof}
Suppose otherwise that there is a $4$-vertex $v$ in $G$. Let $v_1,v_2,v_3,v_4$ be the neighbors of $v$. Let $G'$ be the graph obtained from $G-v$ by adding  $v_1x_1v_2,v_2x_2v_3,v_3x_3v_1$ , where each of $x_i$, $i\in [3]$ is a new $2$-vertex.
By the minimality of $G$, $G'$ has an odd $9$-coloring $c'$.
Then we can get an odd $9$-coloring $c$ of $G$ by coloring each vertex other than $v$ in $G$ with the same color in $G'$ and coloring $v$ with $[9]\setminus \{c(v_1),c(v_2),c(v_3),c(v_4),c_o(v_1),c_o(v_2),c_o(v_3),c_o(v_4)\}$. Since  each of $x_i$ is a $2$-vertex, $c(v_1)\neq(v_2)\neq(v_3)$. Then $v$ has an odd coloring, a contradiction.
\end{proof}

\begin{lem}\label{oddvet-non-adja}


The odd vertex is not adjacent to any odd vertex.

\end{lem}

\begin{proof}
Suppose otherwise that there exist two odd adjacent vertices $u$ and $v$. By Lemma~\ref{minimum degree}, $u$ and $v$ are $5^+$-vertices. Let $G'$ be the graph obtained from $G$ by splitting edge $uv$ with a $2$-vertex $w$. Since $5^+$-vertex $u$ and $v$ have fewer $5^+$-neighbors in $G'$, there is an odd $9$-coloring $c'$ of $G'$ by the minimality of $G$. Note that $c'(u)\neq c'(v)$ since $w$ is a $2$-vertex. Let $c(z)=c'(z)$ for $z\in v(G)$.  Since $u$ and $v$ are odd vertices, $u$ and $v$ always admit an odd coloring. Then $c$ is an odd coloring of $G$, a contradiction.
\end{proof}




\begin{lem}\label{5-vetx}
Let $u$ be a $5$-vertex, $u_1,u_2,\ldots,u_5$ be the neighbors of $u$ in clockwise order, each of $u_2$ and $u_3$ be a $6$-vertex, $[uu_1u_2],[u_2uu_3],[u_3uu_4]$ be  $3$-faces. Then $G$ has no such $5$-vertex $u$.
\end{lem}
\begin{proof}
Suppose otherwise that $G$ has such a $5$-vertex $u$ satisfied these constraints in Lemma~\ref{5-vetx}. Let $u_2',u_2'',u_2'''\notin \{u,u_3,u_1\}$ be the neighbors of {$u_2$; $u_3',u_3'',u_3'''\notin \{u,u_2,u_4\}$ }be the neighbors of $u_3$. Let $G'$ be the graph obtained from $G-\{u_2,u_3\}$ by adding edges between any two of $u_i',u_i'',u_i'''$ if they are not adjacent in $G$ for $i\in\{2,3\}$. Since $G'$ has fewer $4^+$-vertices than $G$,  $G'$ has an odd $9$-coloring of $c'$ by the  minimality of $G$. Let $c(w)=c'(w)$ for $w\in V(G)-\{u_2,u_3\}$.  
Since $c'$ is proper, $c(u_i')\neq c(u_i'') \neq c(u_i''')$  for $i=2,3$.

If  $c(u_1)\notin \{c(u_2'),c(u_2''),c(u_2''')\}$, then $u_2$ must have an odd coloring regardless the colors of $u$ and $u_3$ in $G$. Then   color $u_2$ with the color in $\lbrack 9 \rbrack\setminus\{c(u_2'),c(u_2''),c(u_2'''), c_o(u_2'), c_o(u_2''),c_o(u_2'''),c(u_1)\}$. If either
$\{c(u_2), c(u_4)\}\nsubseteq\{c(u_3'),c(u_3''),c(u_3''')\}$
or $c(u_2)=c(u_4)$, then $u_3$ must have an odd color regardless the color of $u$ in $G$. Then color $u_3$ with the color in $\lbrack 9 \rbrack\setminus\{c(u_3'),c(u_3''),c(u_3'''),\\ c_o(u_3'),c_o(u_3''),c_o(u_3'''),c(u_2),c(u_4)\}$.   Recolor $u$ with the color in $\lbrack 9 \rbrack\setminus\{c(u_1),\\c_o(u_1),c(u_2),c(u_3),c(u_4),c_o(u_4),c(u_5),c_o(u_5)\}$. Since $u$ is a $5$-vertex, $u$ must have an odd color. Then $G$ has an odd $9$-coloring $c$, a contradiction. Thus,
$\{c(u_2), c(u_4)\}\subseteq\{c(u_3'),c(u_3''),c(u_3''')\}$
and $c(u_2)\neq c(u_4)$. We  assume that $c(u_2)=c(u_3'),c(u_4)=c(u_3'')$. In this case, we first  recolor $u$ with the color in $\lbrack 9 \rbrack\setminus\{c(u_1),c_o(u_1),c(u_2),c(u_4),c(u_5),c_o(u_5), c(u_3''')\}$. Then $u_3$ has an odd color $c(u_3''')$. Then color $u_3$ with the color in $\lbrack 9 \rbrack\setminus\{c(u_3'),c(u_3''),c(u_3'''), c_o(u_3'),c_o(u_3''),\\c_o(u_3'''),c(u),c_o(u_4)\}$, a contradiction. 

Thus,  $c(u_1)\in \{c(u_2'),c(u_2''),c(u_2''')\}$. By symmetry, $c(u_4)\in \{c(u_3'),c(u_3''),c(u_3''')\}$. We assume that $c(u_1)=c(u_2'), c(u_4)=c(u_3')$.

If $|\{c(u_2'),c(u_2''),c(u_2'''), c_o(u_2'), c_o(u_2''),c_o(u_2'''),c(u_3''),c(u_3''')\}|\leq 7 $, then color $u_3$ with the color in $\lbrack 9 \rbrack\setminus\{c(u_3'),c(u_3''),c(u_3'''), c_o(u_3'),c_o(u_3''),c_o(u_3'''),c(u_2''),c(u_2''')\}$. Then $u_2$ must have an odd color $c(u_2'')$ or $c(u_2''')$ regardless the color of $u$ in $G$. Then color $u_2$ with the color in $\lbrack 9 \rbrack\setminus\{c(u_2'),c(u_2''),c(u_2'''), c_o(u_2'), c_o(u_2''),c_o(u_2'''),c(u_3''),c(u_3'''), c(u_3)\}$. Since  $|\{c(u_2'),c(u_2''),c(u_2'''),\\ c_o(u_2'),c_o(u_2''),c_o(u_2'''),c(u_3''),c(u_3''')\}|\leq 7 $, $u_2$ has at least one color. Then $u_3$ must have an odd color $c(u_3'')$ or $c(u_3''')$ regardless the color of $u$ in $G$. Finally recolor $u$ with the color in $\lbrack 9 \rbrack\setminus\{c(u_1),c_o(u_1),c(u_2),c(u_3),c(u_4),c_o(u_4),c(u_5),c_o(u_5)\}$, a contradiction. Thus, $|\{c(u_2'),c(u_2''),c(u_2'''), c_o(u_2'), c_o(u_2''),c_o(u_2'''),c(u_3''),c(u_3''')\}|= 8$.

First color $u_2$ with the color in $\lbrack 9 \rbrack\setminus\{c(u_2'),c(u_2''),c(u_2'''), c_o(u_2'), c_o(u_2''),c_o(u_2'''),c(u_3''),c(u_3''')\}$. Then $u_3$ must have an odd color $c(u_3'')$ or $c(u_3''')$ regardless the color of $u$ in $G$.  Let $\{c_1,c_2\}\in \lbrack 9 \rbrack \setminus \{c(u_3'),c(u_3''),c(u_3'''), c_o(u_3'),c_o(u_3''),c_o(u_3'''),c(u_2)\}$. If one of $c_1$ and $c_2$ is not in $\{c(u_2''),c_(u_2''')\}$, then color $u_3$ with this color.  Then $u_2$ must have an odd color $c(u_2'')$ or $c(u_2''')$ regardless the color of $u$.  Recolor  $u$ with the color in $\lbrack 9 \rbrack\setminus\{c(u_1),c_o(u_1),c(u_2),c(u_3),c(u_4),c_o(u_4),\\c(u_5),c_o(u_5)\}$, a contradiction. Thus,  $\{c_1,c_2\}=\{c(u_2''),c(u_2''')\}$. Then color $u_3$ with $c_1$. If $c_2$ is not equal $[9]\setminus\{c(u_1),c_o(u_1),c(u_2),c(u_3),c(u_4),c_o(u_4),c(u_5),c_o(u_5)\}$, then recolor $u$ with the color in $\lbrack 9 \rbrack\setminus\{c(u_1),c_o(u_1),c(u_2),c(u_3),c(u_4),c_o(u_4),c(u_5),c_o(u_5)\}$. Then $u_2$ has an odd color $c_2$,  a contradiction. Thus, $c_2$ is  equal $\lbrack9\rbrack\setminus\{c(u_1),c_o(u_1),c(u_2),c(u_3),c(u_4),c_o(u_4),c(u_5),c_o(u_5)\}$. In this case, recolor $u$ with $c(u_2)$. Since $c_2$ is equal to the only color in  $\lbrack9\rbrack\setminus\{c(u_1),c_o(u_1),c(u_2),\\c(u_3),c(u_4),c_o(u_4),c(u_5),c_o(u_5)\}$, the  color of $u$ does not destroy 
the proper coloring of $u_1,u_3,u_4,u_5$ and the odd coloring of $u_3,u_4,u_5$.
Then we choose one of $c(u_3'')$ and $c(u_3''')$ to color $u_2$ such that $u_1$ has an odd color. Then $u_2$ has an odd coloring $c_2$, $u_3$ has an odd coloring $c(u_3''')$ or $c(u_3'')$, a contradiction.
\end{proof}

A $6$-vertex $v$ is {\em special} if  $v$ is incident with six $3$-faces. A vertex $u$ is {\em free} to the vertex $v$ if $u$ is adjacent to $v$ and $|C_o(u)|\ge 3$. Especially, we use   $\overline{c_o}(u)$ to denote the the only odd color of $u$ if $|C_o(u)|=1$. Note that if $u$ is a $5$-vertex, then $|C_o(u)|=1$, $3$ or $5$. 

In the following Lemmas \ref{32123}-\ref{6-v}, let $u$  be a special $6$-vertex, $u_1,u_2,\ldots,u_6$ be the neighbors of $u$ in the clockwise order,  $v_{12},v_1,v_2 \notin \{u_6,u,u_2\}$ be the neighbors of $u_1$; $v_2,v_3,v_4 \notin \{u_1,u,u_3\}$ be the neighbors of $u_2$; $v_4,v_5,v_6 \notin \{u_2,u,u_4\}$ be the neighbors of $u_3$; $v_6,v_7,v_8 \notin \{u_3,u,u_5\}$ be the neighbors of $u_4$; $v_8,v_9,v_{10} \notin \{u_4,u,u_6\}$ be the neighbors of $u_5$; $v_{10},v_{11},v_{12} \notin \{u_5,u,u_1\}$ be the neighbors of $u_6$; each of $u,u_1,\ldots,u_6,v_1,v_2,\ldots,v_{12}$ is  a special $6$-vertex, as is depicted in Figure 1.
 \vskip 0.5cm
\unitlength=0.25mm
\begin{picture}(10,20)(0,0)
\put(230, -10){\makebox(0,0){$\bullet$}} \put(225, -20) {\scriptsize {\em $u_6$}} \put(230, -10){\line(1,0){40}}   \put(230, -10){\line(2,3){20}}
\put(230, -10){\line(-2,3){20}} \put(230, -10){\line(-1,0){40}}

\put(270, -10){\makebox(0,0){$\bullet$}}  \put(265, -20){\scriptsize {\em $u_1$}} 
 \put(270, -10){\line(2,-3){20}} \put(270, -10){\line(-2,3){20}} 
 \put(270, -10){\line(2,3){20}}  \put(270, -10){\line(1,0){40}} 
 
\put(290, -40){\makebox(0,0){$\bullet$}} \put(278, -38){\scriptsize {\em $u_2$}} 
\put(290, -40){\line(-2,-3){20}} \put(290, -40){\line(2,3){20}}
\put(290, -40){\line(1,0){40}} \put(290, -40){\line(2,-3){20}}

\put(270, -70){\makebox(0,0){$\bullet$}} \put(265, -64){\scriptsize {\em $u_3$}}
\put(270, -70){\line(-1,0){40}} \put(270, -70){\line(1,0){40}}
\put(270, -70){\line(-2,-3){20}}  \put(270, -70){\line(2,-3){20}}

\put(230, -70){\makebox(0,0){$\bullet$}} \put(225, -64){\scriptsize {\em $u_4$}}
\put(230, -70){\line(-2,3){20}} \put(230, -70){\line(2,-3){20}}
\put(230, -70){\line(-2,-3){20}} \put(230, -70){\line(-1,0){40}}

\put(210, -40){\makebox(0,0){$\bullet$}} \put(213, -38){\scriptsize {\em $u_5$}}
\put(210, -40){\line(2,3){20}}  \put(210, -40){\line(-2,3){20}}
\put(210, -40){\line(-2,-3){20}}   \put(210, -40){\line(-1,0){40}}

\put(210,  20){\makebox(0,0){$\bullet$}} \put(210, 23){\scriptsize {\em $v_{11}$}}
\put(210,  20){\line(1,0){40}}  \put(210,  20){\line(-2,-3){20}}

\put(250,  20){\makebox(0,0){$\bullet$}} \put(250, 23){\scriptsize {\em $v_{12}$}}
\put(250,  20){\line(1,0){40}} 

\put(290,  20){\makebox(0,0){$\bullet$}} \put(290, 23){\scriptsize {\em $v_1$}}
\put(290,  20){\line(2,-3){20}} 

\put(310,  -10){\makebox(0,0){$\bullet$}} \put(313, -11){\scriptsize {\em $v_2$}}
\put(310,  -10){\line(2,-3){20}}

\put(330,  -40){\makebox(0,0){$\bullet$}} \put(333, -41){\scriptsize {\em $v_3$}}
\put(330,  -40){\line(-2,-3){20}} 

\put(310,  -70){\makebox(0,0){$\bullet$}} \put(313, -71){\scriptsize {\em $v_4$}}
\put(310,  -70){\line(-2,-3){20}}

\put(290,  -100){\makebox(0,0){$\bullet$}} \put(293, -100){\scriptsize {\em $v_5$}}
\put(290,  -100){\line(-1,0){40}} 

\put(250,  -100){\makebox(0,0){$\bullet$}} \put(253, -98){\scriptsize {\em $v_6$}}
\put(250,  -100){\line(-1,0){40}} 

\put(210,  -100){\makebox(0,0){$\bullet$}} \put(213, -98){\scriptsize {\em $v_7$}}
\put(210,  -100){\line(-2,3){20}}

\put(190,  -70){\makebox(0,0){$\bullet$}} \put(178, -71){\scriptsize {\em $v_8$}}
\put(190,  -70){\line(-2,3){20}}

\put(170,  -40){\makebox(0,0){$\bullet$}} \put(158, -41){\scriptsize {\em $v_9$}}
\put(170,  -40){\line(2,3){20}}

\put(190,  -10){\makebox(0,0){$\bullet$}} \put(174, -11){\scriptsize {\em $v_{10}$}}

\put(250, -40){\makebox(0,0){$\bullet$}} \put(252, -40){\scriptsize {\em $u$}}
\put(250, -40){\line(2,3){20}} \put(250, -40){\line(1,0){40}} 
\put(250, -40){\line(2,-3){20}}  \put(250, -40){\line(-2,-3){20}}
\put(250, -40){\line(-1,0){40}}  \put(250, -40){\line(-2,3){20}}

\put(145, -120){\scriptsize {\em Figure 1. A cluster of special 6-vertices }}

\end{picture}
 \vskip 4cm

In the following Lemmas \ref{32123}-\ref{1234}, let $c$ be an odd $9$-coloring of $G-u$.

\begin{lem}\label{32123}
Let $c(u_1)\neq c(u_2)\neq c(u_3)$, $c(u_2)=c(u_6)$, $c(u_3)=c(u_5)$,  each of $u_1,u_2$ and $u_6$  has exactly one odd color and $\overline{c_o}(u_1)\neq \overline{c_o}(u_2) \neq \overline{c_o}(u_6)$,    $\overline{c_o}(u_i)\notin \{c(u_1),c(u_2),c(u_3),c(u_4)\}$ for each $i\in \{1,2,6\}$.  Then $u_1$ can be recolored
with one color in $\lbrack 9 \rbrack \setminus \{c(u_1),c(u_2), c(u_3), c(u_5)\}$
such that $u_2$ and $u_6$ are free to $u$.

\end{lem}

\begin{proof}
Since $c(u_2)=c(u_6)$ and $u_1$  has exactly one odd color in $G-u$, $c(v_2)=c(v_{12})$, $c(v_1)=\overline{c_o}(u_1)$. Since $c$ is proper, $c(v_{12})=c(v_2)\neq c(u_1)\neq c(u_2)$.  If $c(v_2)=c(v_{12})\neq c(u_3)$, then  $\overline{c_o}(u_2)=c(v_2)=c(v_{12})=\overline{c_o}(u_6)$, a contradiction. Thus, $c(v_2)=c(v_{12})=c(u_3)$. Since each of $u_2$ and $u_6$ has exactly one odd color and $c(u_5)=c(u_3)\neq c(u_1)$, $\{c(v_3),c(v_4)\}=\{\overline{c_o}(u_2),c(u_1)\}$ and  $\{c(v_{10}),c(v_{11})\}=\{\overline{c_o}(u_6),c(u_1)\}$.   Let $c_1$ be the color in  $[9]\setminus\{c(u_1),c(u_2),c(u_3),c(u_4),c(v_1),c_o(v_1),\overline{c_o}(u_2),\overline{c_o}(u_6)\}$.
Recolor $u_1$ with $c_1$. Since $c(u_3)=c(v_2)=c(v_{12})$, $c_1\neq c(v_2)$ and  $c_1\neq c(v_{12})$. Thus, each of $v_1,v_2,v_{12}$ is proper. Observe the neighbors of $v_2$,  $c(v_1)\neq c(u_2)\neq c(v_3)$ since $c(v_3)\in\{c_o(u_2),c(u_1)\}$ and $c(v_1)=\overline{c_o}(u_1)$. Recall that $c_1\notin\{c(v_1),c(u_2),c(u_1),\overline{c_o}(u_2)\}$. Thus, $v_2$ must have an odd coloring in the case of $u_1$ with color $c_1$. By symmetry,  $v_{12}$ must have an odd coloring in the case of $u_1$ with color $c_1$.  Observe the neighbors of $u_2$, $c(u_3)\neq c(v_3) \neq c(v_4)\neq c_1$. Then $|C_o(u_2)|\ge3$ 
regardless of the color of $u$ in $G$. Thus, $u_2$ is free to $u$. By symmetry, $u_6$ is free to $u$.
\end{proof}

\begin{lem}\label{42123}

If $c(u_6)=c(u_2)\neq c(u_1)\neq c(u_3) \neq c(u_5)$, $|C_o(u_1)|=1$,  $|C_o(u_2)|=1$ and $|C_o(u_2)|=1$, then  $\{\overline{c_o}(u_1), \overline{c_o}(u_2), \overline{c_o}(u_6)\}$ occupies at most two different colors in $\lbrack 9 \rbrack \setminus \{c(u_1),c(u_2), c(u_3), c(u_5)\}$ together.
\end{lem}

\begin{proof}

Suppose  otherwise that $\{\overline{c_o}(u_1), \overline{c_o}(u_2), \overline{c_o}(u_6)\}$ occupies three colors in $\lbrack 9 \rbrack \setminus \{c(u_1),c(u_2), \\c(u_3), c(u_5)\}$ together. Thus, $\overline{c_o}(u_1)\neq \overline{c_o}(u_2)\neq \overline{c_o}(u_6)$. 
Since $u_1$ has an odd color and $c(u_2)=c(u_6)$, $c(v_2)=c(v_{12})$ and $c(v_1)=\overline{c_o}(v_1)$. 
If $c(v_2)=c(u_3)$, then $c(v_{12})=c(u_3)$. Since $c(u_3)\neq c(u_5)$, $u_6$ has an odd color $c(v_{12})=c(u_3)$, a contradiction.  
If $c(v_2)=c(u_5)$, then $u_2$ has an odd color $c(v_2)=c(u_5)$ since $c(u_3)\neq c(u_5)$. Then $\overline{c_o}(u_2)=c(u_5)$, a contradiction. Thus, $c(v_2)\notin\{c(u_3),c(u_5)\}$. Then $\{c(v_3),c(v_4)\}=\{c(u_1),c(u_3)\}$ since $u_2$ has  exactly one odd color and $c(u_1)\neq c(u_3)$. Since $u_6$ has  exactly one odd color and $c(u_1)\neq c(u_5)$,   $\{c(v_{10}),c(v_{11})\}=\{c(u_1),c(u_5)\}$. Then $\overline{c_o}(u_2)=c(v_2)=c(v_{12})=\overline{c_o}(u_6)$, a contradiction. 
\end{proof}





\begin{lem}\label{1234}

If $c(u_1)\neq c(u_2)\neq c(u_3)\neq c(u_6)$,  $|C_o(u_1)|=1$ and $|C_o(u_2)|=1$, then  $\{\overline{c_o}(u_1), \overline{c_o}(u_2)\}$ occupies at most one  color in $\lbrack 9 \rbrack \setminus \{c(u_1),c(u_2), c(u_3), c(u_6)\}$.

\end{lem}

\begin{proof}

Suppose otherwise that $\{\overline{c_o}(u_1), \overline{c_o}(u_2)\}$ occupies two  colors in $\lbrack 9 \rbrack \setminus \{c(u_1),c(u_2), c(u_3), c(u_6)\}$. Thus, $\overline{c_o}(u_1)\neq \overline{c_o}(u_2)$. 
Since $u_2$ has an odd color and $c(u_3)\neq c(u_1)$,  $\{\overline{c_o}(u_2), c(u_3), c(u_1)\}=\{c(v_2),c(v_3),c(v_4)\}$. Since $c$ is proper and $v_2u_1\in E(G-u)$, $c(v_2)=\overline{c_o}(u_2)$ or $c(v_2)=c(u_3)$. In the former case, 
since $u_1$ has exactly one odd color and $\overline{c_o}(u_2)\neq c(u_2)\neq c(u_6) $,  $\{c(v_1),c(v_{12})\}=\{c(u_2),c(u_6)\}$.  Then $\overline{c_o}(u_1)=c(v_2)=\overline{c_o}(u_2)$, a contradiction.
In the latter case, $\{c(v_1),c(v_{12})\}=\{c(u_2),c(u_6)\}$ since $u_1$ has exactly one odd color and $c(u_3)\neq c(u_2)\neq c(u_6)$. Then $\overline{c_o}(u_1)=c(v_2)=c(u_3)$, a contradiction.
\end{proof}

\begin{lem}\label{6-v}

$G$ has no configure in Figure 1.

\end{lem}

\begin{proof}

Suppose otherwise that $G$ has the configure in Figure 1. Let $G'$ be the graph obtained from $G-u$ by adding paths $u_1x_1u_3,u_1x_2u_4,u_1x_3u_5$, where $x_1,x_2,x_3$ are $2$-vertices. Since $G'$ has fewer $6^+$-vertices than $G$, $G'$ has an odd $9$-coloring $c'$. Let $c(z)=c'(z)$ for each vertex $z\in V(G)-u$. Since $c'$ is an odd coloring, each of $x_1,x_2$ and $x_3$ has an odd color. Then $c(u_1)\neq c(u_3),c(u_1)\neq c(u_4),c(u_1)\neq c(u_5)$. Then all the color possibilities of $u_1,u_2,\ldots,u_6$ are shown as the following cases by symmetry.

We first establish the following claim:

{\bf Claim} Let $|\{c(u_1),c(u_2),\ldots,c(u_6)\}|=k$. If one of the following statements hold, then $u$ admits an odd coloring in $G$. 

\begin{enumerate}

    \item there exist at least $k-2$ neighbors { $u_i$}, $u_i$ satisfies that $u_i$ are free to $u$ or  { $\overline{c_o}(u_i)$} $\in$ $\{c(u_1),c(u_2),\ldots,c(u_6)\}$;

    \item $|\{\overline{c_o}(u_1),\overline{c_o}(u_2),\overline{c_o}(u_3),
        \overline{c_o}(u_4),\overline{c_o}(u_5),\overline{c_o}(u_6)\}|< 9-k$.

\end{enumerate}

{\bf Case 1} $c(u_1)=1,c(u_2)=c(u_6)=2,c(u_3)=c(u_5)=3,c(u_4)=2$.

By Claim, $\{\overline{c_o}(u_1),\overline{c_o}(u_2),\overline{c_o}(u_3),\overline{c_o}(u_4),
\overline{c_o}(u_5),\overline{c_o}(u_6)\}=\{4,5,6,7,8,9\}$. By Lemma \ref{32123}, we can recolor $u_1$ with color not in $\{c(u_1),c(u_2),c(u_3),c(u_4)\}$ such that $u_2$ and $u_6$ are free to $u$. Then color $u$ with a color in $[9]\setminus\{\overline{c_o}(u_1), \overline{c_o}(u_3),\overline{c_o}(u_4), \overline{c_o}(u_5),c(u_1),c(u_2),c(u_3),c(u_4)\}$.


Since $c(u_1)\neq c(u_4)$, $u$ always admits an odd coloring in $G$. Thus, the odd coloring $c'$ of $G'$ can return back to $G$, a contradiction.

{\bf Case 2} $c(u_1)=1,c(u_2)=c(u_6)=2,c(u_3)=c(u_5)=3,c(u_4)=4$. or $c(u_1)=1,c(u_2)=c(u_5)=3,c(u_6)=c(u_3)=2,c(u_4)=4$.

By claim, $\{5,6,7,8,9\} \setminus\{\overline{c_o}(u_1),\overline{c_o}(u_2),\overline{c_o}(u_3),\overline{c_o}(u_4),\overline{c_o}(u_5),\overline{c_o}(u_6)\}=\varnothing$.  Since $c(u_1)\neq c(u_2)\neq c(u_3)\neq c(u_4)$, 
$\{\overline{c_o}(u_2), \overline{c_o}(u_3)\}$ occupies at most one color in $\{5,\ldots,9\}$ by Lemma \ref{1234}. If $\overline{c_o}(u_3)\in \{5,\ldots,9\}$, then $\{\overline{c_o}(u_1),\overline{c_o}(u_3),\overline{c_o}(u_4),\overline{c_o}(u_5),\overline{c_o}(u_6)\}=\{5,6,7,8,9\}$. If $\overline{c_o}(u_2)\in \{5,\ldots,9\}$, then $\{\overline{c_o}(u_1),\overline{c_o}(u_2),\overline{c_o}(u_4),\overline{c_o}(u_5),\overline{c_o}(u_6)\}=\{5,6,7,8,9\}$. If $\overline{c_o}(u_2), \overline{c_o}(u_3)\notin \{5,\ldots,9\}$, then $u$ admits an odd coloring in $G$ by~Claim.
In each case, $\overline{c_o}(u_5)\neq\overline{c_o}(u_6)$ and each of $\overline{c_o}(u_5)$ and $\overline{c_o}(u_6)$ be in $\{5,6,7,8,9\}$, which contradicts Lemma \ref{1234} since $c(u_1)\neq c(u_6) \neq c(u_5) \neq c(u_4)$.

 {\bf Case 3} $c(u_1)=1,c(u_2)=c(u_6)=c(u_4)=2,c(u_3)=3,c(u_5)=4$.

By claim, $\{5,6,7,8,9\} \setminus\{\overline{c_o}(u_1),\overline{c_o}(u_2),\overline{c_o}(u_3),\overline{c_o}(u_4),\overline{c_o}(u_5),\overline{c_o}(u_6)\}=\varnothing$.  Since $c(u_6)=c(u_2)\neq c(u_1)\neq c(u_3)\neq c(u_5)$, 

$\overline{c_o}(u_i)$ for $i\in\{1,2,6\}$ occupies at most two different colors in $\{5,6,7,8,9\}$  by Lemma \ref{42123}.  If $\{\overline{c_o}(u_1),\overline{c_o}(u_2),\overline{c_o}(u_6)\}$ occupies at most one color in $\{5,6,7,8,9\}$, 
 then $\{5,6,7,8,9\}\setminus\{\overline{c_o}(u_3),\overline{c_o}(u_4),\overline{c_o}(u_5)\}\neq \emptyset$, 
 it contradicts Claim. If $\{\overline{c_o}(u_1),\overline{c_o}(u_2),\overline{c_o}(u_6)\}$ occupies two colors in $\{5, 6, 7,8,9\}$, say $5$ and $6$, then $\{\overline{c_o}(u_3),\overline{c_o}(u_4),\overline{c_o}(u_5)\}= \{7,8,9\}$. If $\{\overline{c_o}(u_2),\overline{c_o}(u_6)\}= \{5,6\}$, then $\{\overline{c_o}(u_6),\overline{c_o}(u_2),\overline{c_o}(u_3),\overline{c_o}(u_4),\overline{c_o}(u_5)\}=\{5,6,7,8,9\}$. If $\{\overline{c_o}(u_1),\overline{c_o}(u_6)\}= \{5,6\}$, then $\{\overline{c_o}(u_1),\overline{c_o}(u_6),\overline{c_o}(u_3),\overline{c_o}(u_4),\overline{c_o}(u_5)\}=\{5,6,7,8,9\}$. If $\{\overline{c_o}(u_1),\overline{c_o}(u_2)\}= \{5,6\}$, then $\{\overline{c_o}(u_1),\overline{c_o}(u_2),\overline{c_o}(u_3),\overline{c_o}(u_4),\overline{c_o}(u_5)\}=\{5,6,7,8,9\}$.  
In the first two cases,  $\overline{c_o}(u_4)\neq \overline{c_o}(u_5)\neq \overline{c_o}(u_6)$ and each of $\overline{c_o}(u_4),\overline{c_o}(u_5)$ and $\overline{c_o}(u_6)$ is in $\{5,6,7,8,9\}$, which contradicts Lemma  \ref{42123} since $c(u_6)=c(u_4)\neq c(u_5) \neq c(u_1) \neq c(u_3)$. In the last case, $\overline{c_o}(u_2)\neq \overline{c_o}(u_3)\neq \overline{c_o}(u_4)$ and each of $\overline{c_o}(u_2),\overline{c_o}(u_3)$ and $\overline{c_o}(u_4)$ is in $\{5,6,7,8,9\}$, which contradicts Lemma  \ref{42123} since $c(u_2)=c(u_4)\neq c(u_1) \neq c(u_3) \neq c(u_5)$.

{\bf Case 4} $c(u_1)=1,c(u_2)=c(u_4)=3,c(u_6)=c(u_3)=2,c(u_5)=4$.

By claim, $\{5,6,7,8,9\} \setminus\{\overline{c_o}(u_1),\overline{c_o}(u_2),\overline{c_o}(u_3),\overline{c_o}(u_4),\overline{c_o}(u_5),\overline{c_o}(u_6)\}=\varnothing$.  Since $c(u_1)\neq c(u_6)\neq c(u_5)\neq c(u_4)$, $\{\overline{c_o}(u_5), \overline{c_o}(u_6)\}$ occupies at most one color in $\{5,\ldots,9\}$ by Lemma~\ref{42123}.
If $\overline{c_o}(u_6)\in \{5,\ldots,9\}$, then $\{\overline{c_o}(u_1),\overline{c_o}(u_2),\overline{c_o}(u_3),\overline{c_o}(u_4),\overline{c_o}(u_6)\}=\{5,6,7,8,9\}$. If $\overline{c_o}(u_5)\in \{5,\ldots,9\}$, then $\{\overline{c_o}(u_1),\overline{c_o}(u_2),\overline{c_o}(u_3),\overline{c_o}(u_4),\overline{c_o}(u_5)\}=\{5,6,7,8,9\}$. If $\overline{c_o}(u_5), \overline{c_o}(u_6)\notin \{5,\ldots,9\}$, then $u$ admits an odd coloring in $G$ by~Claim.


In each case,  $\overline{c_o}(u_2)\neq \overline{c_o}(u_3)\neq \overline{c_o}(u_4)$ and each of $\overline{c_o}(u_2),\overline{c_o}(u_3)$ and $\overline{c_o}(u_4)$ be in $\{5,6,7,8,9\}$, which contradicts Lemma  \ref{42123} since $c(u_2)=c(u_4)\neq c(u_1) \neq c(u_3) \neq c(u_5)$.

{\bf Case 5} $c(u_1)=1,c(u_2)=c(u_6)=2,c(u_3)=3,c(u_5)=4,c(u_4)=5$.

By claim, $\{6,7,8,9\} \setminus\{\overline{c_o}(u_1),\overline{c_o}(u_2),\overline{c_o}(u_3),\overline{c_o}(u_4),\overline{c_o}(u_5),\overline{c_o}(u_6)\}=\varnothing$. By Lemma \ref{1234}, $\{\overline{c_o}(u_2), \overline{c_o}(u_3)\}$ occupies at most one color in $\{6,7,8,9\}$ due to $c(u_1)\neq c(u_2)\neq c(u_3)\neq c(u_4)$; $\{\overline{c_o}(u_5), \overline{c_o}(u_6)\}$ occupies at most one color in $\{6,7,8,9\}$ due to $c(u_1)\neq c(u_6) \neq c(u_5) \neq c(u_4)$. If $\overline{c_o}(u_2), \overline{c_o}(u_5)\in\{6,7,8,9\}$, then $ \{\overline{c_o}(u_1),\overline{c_o}(u_2),\overline{c_o}(u_4),\overline{c_o}(u_5)\}=\{6,7,8,9\}$; if $\overline{c_o}(u_2), \overline{c_o}(u_6)\notin\{6,7,8,9\}$, then $\{\overline{c_o}(u_1),\overline{c_o}(u_2),\overline{c_o}(u_4),\overline{c_o}(u_6)\}=\{6,7,8,9\}$. They contradict Lemma \ref{1234} since $c(u_3)\neq c(u_4)\neq c(u_5)\neq c(u_6)$ and Lemma \ref{42123}   since $c(u_6)=c(u_2)\neq c(u_1)\neq c(u_3)\neq c(u_5)$. If $\overline{c_o}(u_3),\overline{c_o}(u_5)\in\{6,7,8,9\}$, then $ \{\overline{c_o}(u_1),\overline{c_o}(u_3),\overline{c_o}(u_4),\overline{c_o}(u_5)\}=\{6,7,8,9\}$; if $\overline{c_o}(u_3),\overline{c_o}(u_6)\in\{6,7,8,9\}$, then $\{\overline{c_o}(u_1),\overline{c_o}(u_3),\overline{c_o}(u_4),\overline{c_o}(u_6)\}=\{6,7,8,9\}$. They contradict Lemma~\ref{1234} since $c(u_2)\neq c(u_3)\neq c(u_4)\neq c(u_5)$. If $\{\overline{c_o}(u_2),\overline{c_o}(u_3), \overline{c_o}(u_5),\overline{c_o}(u_6)\}$ occupies at most one color in $\{6,7,8,9\}$, then $u$ admits an odd coloring by Claim.

{\bf Case 6} $c(u_1)=1,c(u_2)=3,c(u_6)=(u_3)=2,c(u_5)=4,c(u_4)=5$.

In this case,   $c(u_1)\neq c(u_2)\neq c(u_3)\neq c(u_4)$ and $c(u_1)\neq c(u_6) \neq c(u_5) \neq c(u_4)$. By same argument of Case 5, $\{6,7,8,9\} \setminus\{\overline{c_o}(u_1),\overline{c_o}(u_2),\overline{c_o}(u_4),\overline{c_o}(u_5)\}\\=\varnothing$ or $\{6,7,8,9\} \setminus\{\overline{c_o}(u_1),\overline{c_o}(u_2),\overline{c_o}(u_4),\overline{c_o}(u_6)\}=\varnothing$ or $\{6,7,8,9\} \setminus\{\overline{c_o}(u_1),\\\overline{c_o}(u_3),\overline{c_o}(u_4),\overline{c_o}(u_5)\}=\varnothing$ or $\{6,7,8,9\} \setminus\{\overline{c_o}(u_1),\overline{c_o}(u_3),\overline{c_o}(u_4),\overline{c_o}(u_6)\}=\varnothing$. In the first case, we assume that $\overline{c_o}(u_1)=6,\overline{c_o}(u_2)=7,\overline{c_o}(u_4)=8,\overline{c_o}(u_5)=9$.  Since $\overline{c_o}(u_1)=6$ and $c(u_2)\neq c(u_6)$, $\{6,2,3\}=\{c(v_{12}),c(v_1),c(v_2)\}$. Then $c(v_2)=6$ or $2$. If $c(v_2)=6$, then $\overline{c_o}(u_2)=6$, a contradiction. Thus, $c(v_2)=2$. By symmetry, $c(v_8)=2$.  We  color $u$ with $3$, recolor $u_2\in \lbrack9\rbrack\setminus\{1,2,3,6,7,\overline{c_o}(v_2),\overline{c_o}(v_3),\overline{c_o}(v_4)\}$. Since $\overline{c_o}(u_2)=7$, $c(v_4)=1$ or $7$.  Since $\overline{c_o}(u_4)=8$, $c(v_4)=4$ or $8$. Then $u_3$ has an odd coloring. Since $\overline{c_o}(u_1)=6$, $c(v_{12})=3$ or $6$. Since $\overline{c_o}(u_4)=9$, $c(v_{10})=5$ or $9$. Then $u_6$ has an odd coloring. Then $\overline{c_o}(u_1)=6,\overline{c_o}(u_2)=7,\overline{c_o}(u_4)=8,\overline{c_o}(u_5)=9$. Thus, $G$ has a $9$-odd coloring, a contradiction. In the second case,  $\overline{c_o}(u_1)\neq \overline{c_o}(u_6)$ and each of $\overline{c_o}(u_1)$ and $\overline{c_o}(u_6)$ be in $\{6,7,8,9\}$, which contradicts Lemma  \ref{1234} since $c(u_2)\neq c(u_1) \neq c(u_6) \neq c(u_5)$.  In the last two cases,  $\overline{c_o}(u_3)\neq \overline{c_o}(u_4)$ and each of $\overline{c_o}(u_3)$ and $\overline{c_o}(u_4)$ be in $\{6,7,8,9\}$, which contradicts Lemma  \ref{1234} since $c(u_2)\neq c(u_3) \neq c(u_4) \neq c(u_5)$.

{\bf Case 7} $c(u_1)=1,c(u_2)=3,c(u_6)=2,c(u_3)=4,c(u_5)=4,c(u_4)=5$.

In this case,   $c(u_1)\neq c(u_2)\neq c(u_3)\neq c(u_4)$ and $c(u_1)\neq c(u_6) \neq c(u_5) \neq c(u_4)$. By same argument of Case 5, $\{6,7,8,9\} \setminus\{\overline{c_o}(u_1),\overline{c_o}(u_2),\overline{c_o}(u_4),\overline{c_o}(u_5)\}\\=\varnothing$ or $\{6,7,8,9\} \setminus\{\overline{c_o}(u_1),\overline{c_o}(u_2),\overline{c_o}(u_4),\overline{c_o}(u_6)\}=\varnothing$ or $\{6,7,8,9\} \setminus\{\overline{c_o}(u_1),\\\overline{c_o}(u_3),\overline{c_o}(u_4),\overline{c_o}(u_5)\}=\varnothing$ or $\{6,7,8,9\} \setminus\{\overline{c_o}(u_1),\overline{c_o}(u_3),\overline{c_o}(u_4),\overline{c_o}(u_6)\}=\varnothing$. In first two cases,  $\overline{c_o}(u_1)\neq \overline{c_o}(u_2)$ and each of $\overline{c_o}(u_1)$ and $\overline{c_o}(u_2)$ be in $\{6,7,8,9\}$, which contradicts Lemma  \ref{1234} since $c(u_3)\neq c(u_2) \neq c(u_1) \neq c(u_6)$. In the third case, $\overline{c_o}(u_3)\neq \overline{c_o}(u_4) \neq \overline{c_o}(u_5)$ and each of $\overline{c_o}(u_3), \overline{c_o}(u_4)$ and $\overline{c_o}(u_5)$ be in $\{6,7,8,9\}$, which contradicts Lemma  \ref{42123} since $c(u_5)=c(u_3)\neq c(u_2) \neq c(u_4) \neq c(u_6)$. In last case, $\overline{c_o}(u_1)\neq \overline{c_o}(u_6)$ and each of $\overline{c_o}(u_1)$ and $\overline{c_o}(u_6)$ be in $\{6,7,8,9\}$, which contradicts Lemma  \ref{1234} since $c(u_2)\neq c(u_1) \neq c(u_6) \neq c(u_5)$.

{\bf Case 8} $c(u_1)=1,c(u_2)=3,c(u_6)=c(u_4)=2,c(u_3)=4,c(u_5)=5$.

By claim, $\{6,7,8,9\} \setminus\{\overline{c_o}(u_1),\overline{c_o}(u_2),\overline{c_o}(u_3),\overline{c_o}(u_4),\overline{c_o}(u_5),\overline{c_o}(u_6)\}=\varnothing$.  By Lemma \ref{1234},  $\{\overline{c_o}(u_3),\overline{c_o}(u_4)\}$ occupies at most one color in $\{6,7,8,9\}$ due to $c(u_2)\neq c(u_3)\neq c(u_4)\neq c(u_5)$; $\{\overline{c_o}(u_1) \overline{c_o}(u_6)\}$ occupies at most one color in $\{6,7,8,9\}$ due to $c(u_1)\neq c(u_2) \neq c(u_5) \neq c(u_6)$. Thus, $\{6,7,8,9\} \setminus\{\overline{c_o}(u_1),\overline{c_o}(u_2),\overline{c_o}(u_3),\overline{c_o}(u_5)\}=\varnothing$ or $\{6,7,8,9\} \setminus\{\overline{c_o}(u_6),\overline{c_o}(u_2),\overline{c_o}(u_3),\overline{c_o}(u_5)\}\\=\varnothing$ or  $\{6,7,8,9\} \setminus\{\overline{c_o}(u_1),\overline{c_o}(u_2),\overline{c_o}(u_4),\overline{c_o}(u_5)\}=\varnothing$ or $\{6,7,8,9\} \setminus\{\overline{c_o}(u_6),\\\overline{c_o}(u_2),\overline{c_o}(u_4),\overline{c_o}(u_5)\}=\varnothing$. In the former two cases, they contradict Lemma \ref{1234} since $c(u_1)\neq c(u_2)\neq c(u_3)\neq c(u_4)$. In the latter two cases, they contradict Lemma \ref{1234} due to $c(u_6)\neq c(u_1)\neq c(u_2)\neq c(u_3)$ and Lemma~\ref{42123} due to $c(u_6) = c(u_4)\neq c(u_3)\neq c(u_5)\neq c(u_1)$.

{\bf Case.9} $c(u_1)=1,c(u_2)=2,c(u_6)=6,c(u_3)=3,c(u_5)=5,c(u_4)=4$.

By claim, $\{7,8,9\} \setminus\{\overline{c_o}(u_1),\overline{c_o}(u_2),\overline{c_o}(u_3),\overline{c_o}(u_4),\overline{c_o}(u_5),\overline{c_o}(u_6)\}=\varnothing$.




By Lemma \ref{1234}, $\{u_i, u_{i+1}\}$ occupies at most one color in  $\{7,8,9\}$, where $1\leq i \leq 6$ and $i+1=1$ if $i=6$.
Thus, $\{\overline{c_o}(u_1), \overline{c_o}(u_3), \overline{c_o}(u_5)\}=\{7,8,9\}$ by symmetry.
We assume that,  $\overline{c_o}(u_1)=7,\overline{c_o}(u_3)=8,\overline{c_o}(u_5)=9$ and $c(u_1)\neq c(u_2)\neq \ldots \neq c(u_6)$. Then $c(v_{12})=2$ or $7$, $c(v_2)=6$ or $7$, $c(v_4)=4$ or $8$, $c(v_6)=2$ or $8$, $c(v_8)=9$ or $6$, $c(v_{10})=9$ or $4$. Then color $u$ with $1$. Recolor $u_1\in \lbrack9\rbrack\setminus\{1,2,6,7,c_o(v_1),c_o(v_2),c_o(v_{12})\}$. In each case, $u_2,u_4$ and $u_6$ has at least one odd coloring,  $\overline{c_o}(u_1)=1,\overline{c_o}(u_3)=8,\overline{c_o}(u_5)=9$, and $u$ has the odd coloring, a contradiction.

\end{proof}

Now we are ready to  complete the proof of Theorem~\ref{th1}.   Let each   $v\in V(G)$ have an initial charge of $\mu(v)=d(v)-6$, each $f\in\cup F(G)$ have an initial charge of $\mu(f)=2d(f)-6$.   By Euler's Formula, $|V(G)|+|F(G)|-|E(G)|\geq 0$. Then $\sum_{v\in V }\mu(v)+\sum_{f\in F }\mu(f)=0$.

Let $\mu^*(x)$ denote the final charge of $x\in V(G)\cup F(G)$ after the discharging procedure. To lead to a contradiction, we shall prove that $\sum_{x\in V(G)\cup F(G)}  \mu^*(x)\\ > 0$.  Since the total sum of charges is unchanged in the discharging procedure, this contradiction proves Theorem~\ref{th1}.

We use the following discharging rules:

\begin{enumerate}[(R1)]

 \item  Every $ 4^{+} $-face  sends  $1$ to each incident $5$ -vertex.

    \item  Every $8^{+}$-vertex $u$ sends $\frac{3}{8}$ to each adjacent $5$ -vertex $v$ if both faces incident with the edge $uv$ are $3$-faces. 


\end{enumerate}

By Lemma \ref{minimum degree}, $\delta (G)\geq 5$. Then we check that the final charge of each $5^+$-vertex  and each face is nonnegative.

\begin{enumerate}[1.]

    \item Let $f$ be a $3$-face. By Rules. $f$ is not involved in any discharging procedure. Thus, $\mu^*(f)=2d(f)-6=0$.

    \item Let $f$ be a $4^+$-face. By Lemma \ref{oddvet-non-adja}, $5$-vertex is not adjacent to $5$-vertex. Then $f$ is incident with at most $\frac{d(f)}{2}$ $5$-vertices. By (R1),  $f$ sends $1$ to each incident $5$-vertex. Thus, $\mu^*(f)=2d(f)-6-\frac{d(f)}{2}\geq0$. Note that $\mu^*(f)=0$ if and only if $d(f)=4$ and $f$ is incident with two $5$-vertices.

 \item Let $v$ be a $5$-vertex. If $v$ is incident with at least two $4^+$-faces, then each of $4$-faces sends $1$ to $v$ by (R1). Thus, $\mu^*(v)=d(v)-6+2>0$. If $v$ is incident with one $4$-face, then $v$ is incident with four $3$-faces. Let $v_1,v_2,\ldots,v_5$ be the neighbors of $v$, the face incident with $v_1vv_2$ be a $4^+$-face. By Lemma \ref{oddvet-non-adja}, each of $v_1,\ldots,v_5$ is a $6^+$-vertex. By Lemma \ref{5-vetx} and one of $v_3$ and $v_4$ is a $8^+$-vertex. By (R2), this $8^+$-vertex sends $\frac{3}{8}$ to $v$. Thus, $\mu^*(v)=d(v)-6+1+\frac{3}{8}>0$. Thus, each face incident with $v$ is a $3$-face. By Lemma  \ref{5-vetx}, at least one of $v_{i}$ and $v_{i+1}$ is a $8^+$-vertex where $i\in\{1,2,3,4,5\}$ and $i+1=1$ if $i=5$. Thus, at least three of $v_1,v_2,\ldots,v_5$ are $8^+$-vertices. By (R2), $\mu^*(v)=d(v)-6+\frac{3}{8}\times3>0$.

    \item Let $v$ be a $6$- or $7$-vertex.  By Rules. $v$ is not involved in any discharging procedure. Thus, $\mu^*(v)=d(v)-6=0$ if $f(f)=6$, $\mu^*(v)=d(v)-6>0$ if $f(f)=7$.

    \item Let $v$ be a $8^+$-vertex. By (R2), $v$ sends $\frac{3}{8}$ to adjacent $5$-vertex $v$ if $uv$ is incident with two $3$-faces. Let $v_2$ be the $5$-neighbor of $v$ and $vv_2$ is incident with two $5$-faces $[vv_2v_1]$ and $[vv_2v_3]$. By Lemma \ref{oddvet-non-adja}, $5$-vertex is not adjacent to $5$-vertex. Then each of $v_1$ and $v_3$ are $6^+$-vertices. Thus, $v$ sends charge to at most $\frac{d(v)}{2}$ $5$-neighbors. Thus,  $\mu^*(v)=d(v)-6-\frac{3}{8}\times\frac{d(v)}{2}>0$.

\end{enumerate}

From our hypothesis and the above discharging procedure, we know that the following configurations admit positive final charge and thus can not be contained in $G$. That's, if $G$ has a  $7^+$-vertex or a  $5$-vertex, then $\sum_{x\in V(G)\cup F(G)}  \mu^*(x) > 0$; if $G$ has a   $5^+$-face or a $4$-face incident with at most one $5$-vertex, then $\sum_{x\in V(G)\cup F(G)}  \mu^*(x) > 0$; if $G$ has a $4$-face incident with two $5$-vertices, then these two $5$-vertices has final charge more than $0$, thus $\sum_{x\in V(G)\cup F(G)}  \mu^*(x) > 0$. Thus, $G$ only has $6$-vertices and $3$-faces. By Lemma \ref{6-v}, $G$ has no such configurations, a contradiction.

\end{document}